\input ./style/arxiv-vmsta.cfg
\documentclass[numbers,compress,v1.0.1]{vmsta}

\volume{3}
\issue{2}
\pubyear{2016}
\firstpage{145}
\lastpage{164}
\doi{10.15559/16-VMSTA57}

\startlocaldefs
\allowdisplaybreaks

\newtheorem{thm}{Theorem}

\newtheorem{lem}{Lemma}

\newcommand{\eps}{\varepsilon}
\def\Liminf{\liminf\limits}
\def\Limsup{\limsup\limits}
\renewcommand{\Re}{\mathbb{R}}
\renewcommand{\P}{\mathbf{P}}
\newcommand{\E}{\mathbf{E}}
\newcommand{\X}{\mathbb X}
\newcommand{\XX}{\mathbb X}
\newcommand{\Y}{\mathbb Y}

\endlocaldefs

\begin{document}
\begin{frontmatter}

\title{Large deviation principle for one-dimensional SDEs with
discontinuous coefficients}

\author[a]{\inits{A.}\fnm{Alexei}\snm{Kulik}\corref{cor1}}\email
{kulik.alex.m@gmail.com}
\address[a]{Institute of Mathematics, Ukrai\-ni\-an National Academy of
Sciences,
Tereshchenkivska, 3, Kyiv, 01601, Ukraine}
\cortext[cor1]{Corresponding author.}

\author[b]{\inits{D.}\fnm{Daryna}\snm{Sobolieva}}\email
{dsobolieva@yandex.ua}
\address[b]{Taras Shevchenko National University of Kyiv, Volodymyrska,
64,\\ Kyiv, 01033,
Ukraine}

\markboth{A. Kulik, D. Sobolieva}{Large deviation principle for
one-dimensional SDEs with discontinuous coefficients}

\begin{abstract}
We establish the large deviation principle for solutions of
one-dimensional SDEs with discontinuous coefficients. The main
statement is formulated in a form similar to the classical
Wentzel--Freidlin theorem, but under the considerably weaker assumption
that the coefficients have no discontinuities of the second kind.
\end{abstract}

\received{15 June 2016}
%
\accepted{15 June 2016}
\publishedonline{1 July 2016}
\end{frontmatter}

\section{Introduction and the main result}\label{s1}

This paper aims at the large deviation principle (LDP) for the
solutions to the SDEs
\begin{equation}
\label{sde} dX_t^\eps=a\bigl(X_t^\eps
\bigr)\, dt+\eps\sigma\bigl(X^\eps_t\bigr) dW_t,
\quad X^\eps _0=x_0\in\Re,
\end{equation}
with possibly discontinuous coefficients $a, \sigma$. Recall that a
family of (the distributions of) random elements $ \{X^\eps\}
$ taking values in a Polish space $\XX$ is said to satisfy the \emph
{LDP} with \emph{rate function} $I:\X\to[0,\infty]$ and \emph{speed
function} $r:\Re^+\to\Re^+$ if
\begin{equation}
\label{ldpupper} \Limsup_{\eps\to\infty} r(\eps) \log\P\bigl\{ X^\eps\in F
\bigr\} \leq- \inf_{x \in F} I(x)
\end{equation}
for each closed $F\subset\XX$ and
\begin{equation}
\label{ldplower} \Liminf_{\eps\to\infty} r(\eps) \log\P\bigl\{ X^\eps\in G
\bigr\}\geq- \inf_{x \in G} I(x)
\end{equation}
for each open $G\subset\XX$. The rate function is assumed to be lower
semicontinuous; that is, all level sets $ \{ x: I(x) \leq c
\}$, $ c\geq0$, are closed. If all level sets are compact, then
the rate function is called \emph{good}.

We assume that, for some $C, c>0$,
\begin{equation}
\label{coef}
\bigl|a(x)\bigr|\leq C\bigl(1+|x|\bigr),\qquad c\leq\sigma^2(x)\leq C,\quad
x\in\Re.
\end{equation}
It is well known that, in this case, the SDE \eqref{sde} has a unique
weak solution, which can be obtained by a proper combination of the
time change transformation of a Wiener process and the Girsanov
transformation of the measure; see \cite{watanabe-ikeda}, IV, \S4. In
what follows, we fix $T>0$, interpret the (weak) solution $X^\eps= \{
X_t^\eps, t\in[0, T]\}$ to \eqref{sde} as a random element in $C(0, T)
$, and prove the LDP for the family $\{X^\eps\}$.
Since the law of $X^\eps$ does not depend on a possible change of the
sign of $\sigma$, in what follows, we assume without loss of generality
that $\sigma>0$.

Our principal regularity assumptions on the coefficients $a, \sigma$
is that they have no discontinuities of the second kind, that is, they
have left- and right-hand limits at every point $x\in\Re$. For a given
pair of such functions $a, \sigma$, we define the modified
functions $\bar a, \bar\sigma$ as follows:
\begin{itemize}
\item[(i)] if $a(x-)\geq0$ and $a(x+)\leq0$, then $\bar a(x)=0$ and
$\bar\sigma(x)=\sigma(x)$;
\item[(ii)] otherwise, $\bar a(x), \bar\sigma(x)$ equal either $\bar
a(x-), \bar\sigma(x-)$ or $\bar a(x+), \bar\sigma(x+)$ with the
choice made in such a way that
\[
{\bar a^2(x)\over\bar\sigma^2(x)}=\min \biggl({a^2(x-)\over\sigma
^2(x-)},
{a^2(x+)\over\sigma^2(x+)} \biggr).
\]
\end{itemize}

Denote by $\mathit{AC}(0, T)$ the class of absolutely continuous functions $\phi
:[0,T]\to\Re$, and for each $f\in \mathit{AC}(0, T)$, we denote by $\dot f$ its
derivative, which is well defined for almost all $t\in[0, T]$.

\begin{thm}\label{t1} Let $a, \sigma$ satisfy \eqref{coef} and have no
discontinuities of the second kind.
Then the family of distributions in $C(0,T)$ of the solutions to SDEs
\eqref{sde} satisfies the LDP with the speed function $r(\eps)=\eps^2$
and the good rate function equal
to
\[
I(f)= {1\over2}\int_0^T
{(\dot f_t-\bar a(f_t))^2\over\bar\sigma
^{2}(f_t)}\, dt
\]
for $f\in \mathit{AC}(0,T)$ with $f_0=x_0$ and $I(f)=\infty$ otherwise.
\end{thm}

Theorem \ref{t1} has the form very similar to the classical
Wentzel--Freidlin theorem (\cite{freidlin-wentzell}, Chapter 3, \S2),
which establishes LDP in a much more general setting, for
multidimensional Markov processes that may contain both diffusive and
jump parts. However, the Wentzel--Freidlin approach substantially
exploits the~continuity of infinitesimal characteristics of the
process. The natural question arises: to which extent the continuity
assumption can be relaxed in this theory. In \cite
{chiang-sheu1,chiang-sheu2}, the LDP was established for
multidimensional diffusions with unit diffusion matrix and drift
coefficients discontinuous along a given hyperplane; see also \cite
{alanyali-hajek1,alanyali-hajek2,chiang-sheu3} for some other results
in this direction. In \cite{krykun}, this result, in the
one-dimensional setting, is extended to the case of piecewise smooth
drift and diffusion coefficients with one common discontinuity point.
The technique in the aforementioned papers is based on the analysis of
the joint distribution of the process itself and its occupation time in
the half-space above the discontinuity point (surface) and is hardly
applicable when the structure of the discontinuity sets for the
coefficients is more complicated. In \cite{kulik-sobolieva}, the LDP
for a one-dimensional SDE with zero drift coefficient was established
under a very mild regularity condition on $\sigma$: for the latter, it
was only assumed that its discontinuity set has zero Lebesgue measure.
Extension of this result to the case of nonzero drift coefficient is
far from being trivial. In \cite{sobolieva}, such an extension was
provided, but the assumption therein that $a/\sigma^2$ possesses
a~bounded derivative is definitely too restrictive. In this paper, we
summarize the studies from \cite{kulik-sobolieva} and \cite{sobolieva};
note that the assumption on $\sigma$ in the current paper is slightly
stronger than in~\cite{kulik-sobolieva}.\looseness=1

We note that our main result, Theorem \ref{t1}, well illustrates the
relation of the LDP with discontinuous coefficients to the classical
Wentzel--Freidlin theory: the rate function in this theorem is given in
a classical form, but with the properly modified coefficients. The
heuristics of this modification is clearly seen. Namely, thanks to
(ii), the rate functional $I$ is lower semicontinuous; see Section \ref
{s22}. Assertion (i) corresponds to the fact that, in the case
$a(x-)\geq0$ and $a(x+)\leq0$, the family $X^\eps$ with $X^\eps_0=x$
weakly converges to the constant function equal to $x$. We interpret
the limiting function as the solution to the ODE $\dot x_t=\bar
a(x_t)$, and note that a similar ODE for $a$ may fail to have a
solution at all.

\section{Preliminaries to the proof}\label{s2}

\subsection{Exponential tightness, contraction principle}\label{s21}
Recall that a family $\{X^\eps\}$ is called \emph{exponentially tight}
with the speed function $r(\eps)$ if for each $Q>0$, there exists a
compact set $K\subset\XX$ such that
\[
\limsup_{\eps\to0}r(\eps)\log\P \bigl(X^\eps\not\in K
\bigr)\leq-Q.
\]
For an exponentially tight family, the LDP is equivalent to the \emph
{weak} LDP; the latter by definition means that the upper bound \eqref
{ldpupper} holds for all \emph{compact} sets~$F$, whereas the lower
bound \eqref{ldplower} still holds for all open sets $G$. An equivalent
formulation of the weak LDP is the following: for each
$x\in\XX$,\vadjust{\goodbreak}
\begin{align}
&{}\lim_{\delta\to0}\limsup_{\eps\to0}r(\eps)
\log\P \bigl(X^\eps\in B_\delta(x) \bigr) 
\nonumber
\\
&{}\quad =\lim_{\delta\to0}\liminf_{\eps\to0}r(\eps)\log\P
\bigl(X^\eps\in B_\delta(x) \bigr)=-I(x),\label{wLDP}
\end{align}%
where $B_\delta(x)$ denotes the open ball with center $x$ and radius
$\delta$.

To prove \eqref{wLDP}, we will use a certain extension of the \emph
{contraction principle}, which in its classical form (e.g., \cite
{feng-kurtz}, Section 3.1, and \cite{dembo-zeitouni}, Section 4.2.1)
states the LDP for a family $X^\eps=F(Y^\eps)$, where $Y^\eps$ is a
family of random elements in a Polish space $\mathbb{Y}$ that satisfies
an LDP with a good rate function $J$, and $F:\mathbb{Y}\to\X$ is a
\emph{continuous} mapping. The rate function for $X^\eps$ in this case
has the form
\[
I(x)=\inf_{y: F(y)=x}J(y).
\]
In the sequel, we use two different representations of our particular
family $X^\eps$ as an image of certain family whose LDP is well
understood; however, the functions $F$ in these representations fail to
be continuous. Within such a~framework, the following general lemma
appears quite useful. We denote by $\rho_\X, \rho_\Y$ the metrics in
$\XX,\Y$ and by $\varLambda_F$ the set of continuity points of a mapping
$F:\Y\to\X$. Note that $\varLambda_F$ is Borel measurable; see Appendix
II in \cite{billingsley}.

\begin{lem}\label{l1} Let family $Y^\eps$ satisfy the LDP with speed
function $r(\eps)$ and rate function~$J$. Assume also that
\[
\P \bigl(Y^\eps\in\varLambda_F \bigr)=1, \quad\eps>0.
\]
Then, for any $x\in\XX$,
\begin{equation}
\label{scup} \lim_{\delta\to0}\limsup_{\eps\to0}r(\eps)
\log\P\bigl(X^\eps\in B_\delta (x)\bigr)\leq-I_{\mathit{upper}}(x),
\end{equation}
\begin{equation}
\label{sclow} \lim_{\delta\to0}\liminf_{\eps\to0}r(\eps)
\log\P\bigl(X^\eps\in B_\delta (x)\bigr)\geq-I_{\mathit{lower}}(x)
\end{equation}
with
\[
I_{\mathit{upper}}(x)=\lim_{\delta\to0}\lim_{\gamma\to0}\inf
_{y\in\varXi_{\gamma
,\delta}(x)}J(y), \qquad I_{\mathit{lower}}(x)=\lim_{\delta\to0}
\inf_{y\in\varTheta
_\delta(x)}J(y),
\]
where
\begin{align*}
\varTheta_\delta(x)&{}= \bigl\{y\in\varLambda_F:\rho_X \bigl(x, F(y) \bigr)<\delta \bigr\},\\
\varXi_{\gamma, \delta}(x)&{}= \bigl\{y\in\Y: \rho_\Y \bigl(y, \varTheta_\delta(x) \bigr)<\gamma \bigr\}.
\end{align*}
\end{lem}
\begin{proof} We have
\[
\P \bigl(X^\eps\in B_\delta(x) \bigr)=\P \bigl(X^\eps
\in B_\delta(x), Y^\eps\in\varLambda_F \bigr)=\P
\bigl(Y^\eps\in\varTheta_\delta(x) \bigr).
\]
Thus, the upper bound in the LDP for $\{Y^\eps\}$ gives
\[
\limsup_{\eps\to0}r(\eps)\log\P \bigl(X^\eps\in
B_\delta(x) \bigr)\leq-\inf \bigl\{ J(y), y\in\bar\varTheta_\delta(x)
\bigr\},
\]
where $\bar\varTheta_\delta(x)$ denotes the closure of $\varTheta_\delta(x)$.
Since $\bar\varTheta_\delta(x)\subset\varXi_{\gamma, \delta}(x)$ for any
$\gamma>0$, this provides \eqref{scup}. The proof of \eqref{sclow} is
even simpler: for any $y\in\varTheta_\delta(x)$, there exists $r>0$
such that
the image of the ball $B_r(y)$ under $F$ is contained in $B_\delta(x)$,
which yields
\[
\limsup_{\eps\to0}r(\eps)\log\P \bigl(X^\eps\in
B_\delta(x) \bigr)\leq\limsup_{\eps\to0}r(\eps)\log\P
\bigl(Y^\eps\in B_r(y) \bigr)\geq-J(y).\qedhere
\]
\end{proof}

Lemma \ref{l1} is a simplified and more precise version of Lemma 4 in
\cite{kulik}. The functions $I_{\mathit{upper}}$, $I_{\mathit{lower}}$ are lower
semicontinuous: we can show this easily using that, for any sequence
$x_n\to x$, the sets $\varTheta_{\delta/2}(x_n)$ are embedded into
$\varTheta
_{\delta}(x)$ for~$n$ large enough (see, e.g., Proposition 3 in \cite{kulik}).
In fact, Lemma \ref{l1} says that for an arbitrary image of a family
$\{
Y^\eps\}$, one part of an LDP (the upper bound) holds with one rate function,
whereas the other part (the lower bound) holds with another rate
function. This is our reason to call \eqref{scup} and \eqref{sclow}
the \emph{upper} and the \emph{lower} \emph{semicontraction
principles}. To prove \eqref{wLDP}, it suffices to verify the inequalities
\begin{equation}
\label{verify} I_{\mathit{lower}}(x)\leq I(x), \qquad I(x)\leq I_{\mathit{upper}}(x),
\quad x\in\XX.
\end{equation}
We refer to \cite{kulik} for a more discussion and an example where
the pair
of semicontraction principles do not provide an LDP.

\subsection{Lower semicontinuity of $I$}\label{s22} In this section, we
prove directly that the functional $I$ specified in Theorem~\ref{t1} is
lower semicontinuous, that is, it is indeed a rate functional. This
will explain the particular choice of the modified functions $\bar a,
\bar\sigma$. In addition, this will simplify the proofs, where we will
use the representation for $I(x)$ presented further.

Define
\[
S(x)=\int_0^x {\bar a(z)\over\bar\sigma^{2}(z)}\, dz,
\quad x\in\Re.
\]
Then $I(f)$, if it is finite, can be represented as
\begin{align}
I(f)&{}={1\over2}\int_0^T
{(\dot f_t)^2\over\bar\sigma^{2}(f_t)}\, dt+ \bigl[S(f_T)-S(f_0) \bigr]
+{1\over2}\int_0^T
{\bar a^2(f_t)\over\bar\sigma^{2}(f_t)}\, dt
\nonumber
\\
&{}
=:I_1(f)+I_2(f)+I_3(f).\label{repI}
\end{align}
The function $S$ is continuous; hence, the functional $I_2$ is just
continuous. The function $\bar a^2/\bar\sigma^2$ is lower
semicontinuous by the choice of $\bar a, \bar\sigma$; thus, the
functional $I_3$ is lower semicontinuous. Finally, we can represent
$I_1$ in the form
$
I_1(f)=I_0(\varSigma(f_\cdot)),
$
where the function
\[
\varSigma(x)=\int_0^x{1\over\bar\sigma(z)}\, dz
\]
is continuous, and the functional
\[
I_0(f)={1\over2}\int_0^T
(\dot f_t)^2\, dt
\]
is known to be lower semicontinuous (this is just the rate functional
for the family $\{\eps W\}$). Hence, $I_1$ is lower semicontinuous,
which completes the proof of the statement.

\section{Proof of Theorem \ref{t1}}

\subsection{Exponential tightness and the weak LDP}\label{s31} In this
section, we prove that the family $\{X^\eps\}$ is {exponentially tight}
with the speed function $r(\eps)=\eps^2$. Note that
\[
M^\eps_t:=\int_0^t
\sigma^2 \bigl(X_s^\eps \bigr)\,
dW_s
\]
is a continuous martingale with the quadratic characteristics
\[
\bigl\langle M^\eps \bigr\rangle_t=\int
_0^t\sigma^2 \bigl(X_s^\eps
\bigr)\, ds\leq Ct;
\]
see \eqref{coef}. Recall that $M^\eps$ can be represented as a Wiener
process with the time change $t\mapsto\langle M^\eps\rangle_t$; see,
for example, \cite{watanabe-ikeda}, II. \S7. Then, for each $R$,
\begin{align*}
&\limsup_{\eps\to0}\eps^2\log\P \Bigl(\sup
_{t\in[0, T]}\eps\bigl|M^\eps_t\bigr|>R \Bigr)
\\
&\quad{}\leq\limsup_{\eps\to0}\eps^2\log\P \Bigl(\sup
_{t\in[0, CT]}\eps|W_t|>R \Bigr)=-{R^2\over2CT}.
\end{align*}
On the other hand, for each $\omega\in\varOmega$ such that $\eps
|M^\eps_t(\omega)|>R$, the corresponding trajectory of $X^\eps$ satisfies
\[
\bigl|X^\eps_t\bigr|\leq|x_0|+C\int_0^t
\bigl(1+\bigl|X_s^\eps\bigr| \bigr)\, ds+R,
\]
and therefore, by the Gronwall inequality,
\[
\sup_{t\in[0,T]}\bigl|X^\eps_t\bigr|
\leq\bigl(|x_0|+CT+R\bigr)e^{CT}.
\]
Therefore, for any $Q>0$, there exists $R$ such that
\[
\limsup_{\eps\to0}\eps^2\log\P \Bigl(\sup
_{t\in[0, T]}\bigl|X^\eps_t\bigr|>R \Bigr)\leq-Q.
\]

Next, recall the Arzel\`a--Ascoli theorem: for a closed set $K\subset
C(0,T)$ to be compact, it is necessary and sufficient that it is
bounded and equicontinuous. The family $\eps M^\eps$ is represented as
a time changed family $\eps W^\eps$, where each $W^\eps$ is a Wiener
process, and the derivative of $\langle M^\eps\rangle_t$ is bounded by
$C$. Using these observations, it is easy to deduce the exponential
tightness for $\{\eps M^\eps\}$ using the well-known fact that the
family $\{\eps W\}$ is exponentially tight. On the other hand, for any
$\omega$ such that the trajectory of $X^\eps_t$ is bounded by $R$, the
corresponding trajectory of the process $X^\eps_t-\eps M^\eps_t$
satisfies the Lipschitz condition w.r.t.\ $t$ with the constant
$C(1+R)$. Combined with the previous calculation, this easily yields
the required exponential tightness.

In what follows, we proceed with the proof of \eqref{wLDP}. Since now
the state space $\XX=C(0, T)$ is specified, we change the notation and
denotes the points in this space by $f, g, \dots$. Since the set
$B_1(f)$ is bounded, the law of $X^\eps$ restricted to any $B_\delta
(f)$ does not change if we change the coefficients $a, \sigma$ on the
intervals $(-\infty, -R], [R, \infty)$ with $R>0$ large enough. Hence,
we furthermore assume the coefficients $a, \sigma$ to be constant on
such intervals for some $R$.

\subsection{Case I. Piecewise constant $a,\sigma$ with one
discontinuity point} We proceed with the further proof in a
step-by-step way, increasing gradually the classes of the coefficients
$a, \sigma$ for which the corresponding LDP is proved. First, let
$a,\sigma$ be constant on the intervals $(-\infty, z)$ and $(z,
\infty
)$ with some $z\in\Re$. Without loss of generality, we can assume that
$z=0$. Then we can use Theorem 2.2 \cite{krykun}, where the LDP with
the speed function $r(\eps)=\eps^2$ is established for the pair
$(X^\eps
, Z^\eps)$ with
\[
Z_t^\eps=\int_0^t1_{(0, \infty)}
\bigl(X^\eps_s \bigr)\, ds, \quad t\in[0, T].
\]
The corresponding rate function in \cite{krykun} is given in the
following form. Denote $a_\pm=a(0\pm), \sigma_\pm=\sigma(0\pm)$ and
define the class $H(f)$ of functions
$\psi\in \mathit{AC}(0, T)$ such that
\[
\dot\psi_t\quad\left\{ %
\begin{array}{ll}
=0, & f_t<0; \\
=1, & f_t>0; \\
\in[0,1],& f_t=0.
\end{array}\right. %
\]
Then the rate functional for $(X^\eps, Z^\eps)$ equals
\[
I(f, \psi)={1\over2}\int_0^T
L(f_t, \dot f_t, \phi_t)\, dt
\]
with
\[
L(x, y, z)= \left\{ %
\begin{array}{ll}
{(y -a(x))^2\over\sigma^2(x)}, & x\not=0; \\
{(a_+ z+a_-(1-z))^2\over\sigma^2_+z+\sigma_-^2 (1-z)}, & x=0, \,
{a_-\over\sigma^2_-}>{a_+\over\sigma^2_+};\\
{a_+^2 \over\sigma^2_+}z+{a_-^2 \over\sigma^2_-}(1-z), & x=0, \,
{a_-\over\sigma^2_-}\leq{a_+\over\sigma^2_+}
\end{array} \right.%
\]
for all pairs $(f, \psi)$ such that $f\in \mathit{AC}(0, T), f_0=x_0, \psi\in
H(f)$ and, for all other pairs, $I(f, \psi)=\infty$.

From this result, using the {contraction principle} (see Section \ref
{s2}), we easily derive the LDP for $X^\eps$ with the rate function
\[
I(f)=\inf_{\psi\in H(f)}I(f, \psi)=\int_0^T
L(f_t, \dot f_t)\, dt, \qquad L(x, y)=\inf
_{z\in[0,1]}L(x, y,z)
\]
for $f\in \mathit{AC}(0, T), f_0=x_0$ and $I(f)=\infty$ otherwise. Now only a
minor analysis is required to show that this rate function actually
coincides with that specified in Theorem \ref{t1}. First, we observe that
\[
L(0,y)= {\bar a^2(0)\over\bar\sigma^{2}(0)}.
\]
This is obvious if either $a_-/\sigma^2_-\leq a_+/\sigma^2_+$ or
$a_->0, a_+<0$. In the case where $a_-/\sigma^2_->a_+/\sigma^2_+$ and
$a_-, a_+$ have the same sign, we can verify directly that\break
$L'_z(0,y,z)$ have the same sign for $z\in[0,1]$, which completes the
proof of the required identity.

We will use repeatedly the following fact, which follows easily from
the change-of-variables formula: for any $f\in \mathit{AC}(0,T)$ and any set
$A\subset\Re$ with zero Lebesgue measure, the Lebesgue measure of the set
\begin{equation}
\label{zero} \int_0^T 1_{f_t\in A, \dot f_t\not=0}\, dt=0;
\end{equation}
see, for example, Lemma 1 in \cite{kulik-sobolieva}. Applying \eqref
{zero} with $A=\{0\}$, we conclude that in the above expression for
$I(f)$, the function $L$ can be changed to
\[
L(x,y)={(y-\bar a(x))^2\over\bar\sigma^{2}(x)},
\]
which completes the proof of Theorem \ref{t1} in this case.

\subsection{Case II. Piecewise constant $a,\sigma$} Let, for some
$z_1<\dots<z_m$, the functions $a, \sigma$ be constant on the intervals
$(-\infty, z_1)$, $(z_{1}, z_2),$ $\dots, (z_m, \infty)$. Assume that
$x_0\not\in\{z_k, k=1, \dots, m\}$, which does not restrict the
generality of the construction given further, and define the functions
$a_k, \sigma_k, k=0,\dots, m$ by
\[
a_0(x)=a(x_0),\qquad\sigma_0(x)=
\sigma(x_0), \quad x\in\Re,
\]
\begin{align*}
a_k(x)&{}=\left \{ %
\begin{array}{ll}
a(z_k-), & x<z_k, \\
a(z_k+), & x\geq z_k,
\end{array} \right.%
\\
\sigma_k(x)&{}= \left\{ %
\begin{array}{ll}
\sigma(z_k-), & x<z_k, \\
\sigma(z_k+), & x\geq z_k,
\end{array}\right. %
\quad
k=1, \dots,m.
\end{align*}
Consider a family of independent processes $Y^{0, \eps}, Y^{n,k, \eps},
k=1, \dots, m, n\geq1$, such that $Y^{0, \eps}$ solves SDE \eqref{sde}
with the coefficients $a_0, \sigma_0$ and each $Y^{n,k, \eps}$ solves a
similar SDE with the coefficients $a_k, \sigma_k$ and the initial value
$z_k$. Define iteratively the process $\tilde X^\eps$ in the following
way: put $\tilde X^\eps$ equal $Y^{0, \eps}$ until the time moment
\[
\tau_1:=\inf \bigl\{t:Y^{0, \eps}_t\in
\{z_k, k=1, \dots, m\} \bigr\}.
\]
Define the random index $\kappa_1\in\{1, \dots, m\}$ such that $Y^{0,
\eps}_{\tau_1}=z_{\kappa_1}$. Then put $\tilde X^\eps_t=Y^{1,\kappa_1,
\eps}_{t-\tau_1}$ until the first time moment $\tau_2$ when this
process hits $\{z_k, k=1, \dots, m\}\setminus\{z_{\kappa_1}\}$.
Iterating this procedure, we get a process $\tilde X^\eps_t$ with
\[
X_t^\eps=Y^{0, \eps}_t, \quad t\leq
\tau_1, \qquad X_t^\eps=Y^{n,\kappa
_k, \eps}_{t-\tau_n},
\quad t\in[\tau_n, \tau_{n+1}],\  n\geq1.
\]
It follows from the strong Markov property of $X^\eps$ that $\tilde
X^\eps$ has the same law with $X^\eps$. Hence, the given construction
in fact represents the law of $X^\eps$ as the image of the joint law
family of independent processes $Y^{0, \eps}, Y^{n,k, \eps}, k=1,
\dots
, m, n\geq1$. Each of these processes is a solution to \eqref{sde}
with corresponding coefficients having at most one discontinuity point;
hence, the LDP for them is provided in the previous section. Our idea
is to deduce the LDP $X^\eps$ via a~version of
the contraction principle. With this idea in mind, we first perform a
simplification of the above representation. For some $N$ (the choice of
$N$ will be discussed below), we consider the space $\mathbb
{Y}=C(0,T)^{1+mN}$ and construct a function $F:\mathbb{Y}\to\X=C(0,T)$
in the following way. For $y=(y^0, y^{n,k}, k=1, \dots,m, n=1, \dots,
N)$, we first define $\tau_1(y)=\inf\{t: y^0_t\in\{z_k, k=1, \dots,
m\}
\}$ with the usual convention that $\inf\varnothing=T.$ The function
$[F(y)]_t, t\in[0, T]$, is defined to be equal to $y^0_t$ for $t\leq
\tau_1(y)$. If $\tau_1(y)<T$, then the construction is iterated: we
define $\kappa_1(y)$ by $y^0_{\tau_1(y)}=x_{\kappa_1(y)}$ and put, for
$t\geq\tau_1(y)$, $[F(y)]_t$ equal to $y^{1,\kappa_1(y)}_{t-\tau
_1(y)}$ up to the first moment when this function hits $\{z_k, k=1,
\dots, m\}\setminus\{x_{\kappa_1(y)}\}$. We iterate this procedure at
most $N$ times; that is, if $\tau_N(y)<T$, then we put
\[
\bigl[F(y) \bigr]_t=y^{N,\kappa_N(y)}_{t-\tau_N(y)}, \quad t\in
\bigl[ \tau_N(y), T \bigr].
\]
Denote
\[
\varDelta:=\min_{k=2, \dots, m}(z_k-z_{k-1}).
\]
For any fixed $f\in C(0,T)$, we can choose $\delta_f>0$ small enough
and $N_f$ large enough so that each $g\in B_\delta(f)$ has less than
$N$ $\varDelta$-oscillations on $[0,T]$. Hence, if in the above
construction, $N$ is taken equal to $N_f$, then the restriction of the
law of $X^\eps$ to any ball $B_\delta(f), \delta\leq\delta_f$, equals
to the same restriction of the image of the joint law of the finite
family $Y^{0, \eps}, Y^{n,k, \eps}, k=1, \dots, m, n=1, \dots, N$,
under the mapping $F$ specified before.

We aim to verify \eqref{wLDP}, and we argue in the following way. We
fix $f$ and choose $N=N_f$ as before, so that the laws of $X^\eps$,
restricted to $B_\delta(f)$ for $\delta$ small, can be obtained as the
image under $F$ specified before. Then we prove \eqref{verify} at this
\emph{particular} point $x=f$, with $I_{\mathit{lower}}, I_{\mathit{upper}}$ being
constructed by this \emph{particular} $F$. This yields the required
weak LDP \eqref{wLDP}.

Within such an argument, we have to treat for any $N$ the image under
the corresponding $F$ of the family of laws in $\mathbb
{Y}=C(0,T)^{1+mN}$, which, according to the result proved in the
previous section, satisfies the LDP with the rate function
\[
J(y)={1\over2} \int_0^T
{(\dot y^0_t- a_0(y_t^0))^2\over\sigma
^{2}_0(y_t^0)}\, dt+{1\over2}\sum
_{k=1}^m\sum_{n=1}^N
\int_0^T{(\dot y^{k,n}_t- \bar a_k(y^{k,n}_t))^2\over\bar\sigma
^{2}_k(y_t^{k,n})}\, dt
\]
for
\[
y= \bigl(y^0, y^{n,k} \bigr)_{k\leq m, n\leq N}
\]
such that $y^0, y^{n,k}\in \mathit{AC}(0,T)$, $y^0_0=x_0, y^{n,k}_0=z_k$ and
$J(y)=\infty$ otherwise. To apply Lemma \ref{l1} in this setting, we
first analyze the structure and the properties of the corresponding $F$.

Each trajectory $f=F(y)\in C(0,T)$ is actually a patchwork, which
consists of pieces of trajectories $y^0, y^{n,k}, k=1, \dots,m, n=1,
\dots, N$: the pasting points are $\tau_1(y), \dots, \tau_r(y),$
$r=r(y)\leq N$, and after $\tau_n(y)$, the (part of the) new trajectory
is used with the number $\kappa_n(y)$. For a $y_l\to y$ in $\Y$, the
corresponding sequence of trajectories $f_l=F(y_l)$ may fail to
converge to $f$ because the functionals $\tau_n(\cdot), \kappa
_n(\cdot
)$ are not continuous. However, the above ``patchwork representation''
easily yields the following two facts:
\begin{itemize}
\item any limit point $f_*$ of the sequence $\{f_l\}$ possesses a
similar representation with the same $y=\lim_l y_l$ and with the
corresponding pasting points
$\tau_1^*, \tau_2^*, \dots$ and numbers $\kappa_1^*, \kappa_2^*,
\dots$
being partial limits of the sequences $\{\tau_1(y_l)\}, \{\tau
_2(y_l)\}
, \dots$ and $\{\kappa_1(y_l)\}, \{\kappa_2(y_l)\}, \dots$;
\item if the functions $\tau_1(\cdot), \tau_2(\cdot), \dots$ are
continuous at a given point $y\in\Y$, then $y\in\varLambda_F$.
\end{itemize}

Using the first fact, now it is easy to prove the second inequality in
\eqref{verify}. If it fails, then for a given $f$, there exists a sequence
$\{y^l\}$ such that $F(y_l)\to f$ and $J(y_l)\leq c<I(f)$. Since the
level set $\{y:J(y)\leq c\}$ is compact, we can assume without loss of
generality that $y_l$ converge to some $y$; recall that $J$ is lower
semicontinuous and thus $J(y)\leq c$. The function $f$ possesses the
above patchwork representation with the trajectories taken from $y$,
some pasting points $\tau_1^*, \dots,\tau_r^*$, and some numbers
$\kappa
_1^*, \dots,\kappa_r^*$. From this representation it is clear that
$f\in \mathit{AC}(0,T)$ and $f_0 =x_0$: if this fails, then the same properties
fail at least for one trajectory from the family $y$ and thus
$J(y)=\infty$, which contradicts to $J(y)\leq c$. Hence, we have
\[
I(f)={1\over2}\int_0^T
{(\dot f_t-\bar a(f_t))^2\over\bar\sigma
^{2}(f_t)}\, dt={1\over2}\sum
_{n=1}^{r+1}\int_{\tau^*_{n-1}}^{\tau
^*_{n}}
{(\dot f_t-\bar a(f_t))^2\over\bar\sigma^{2}(f_t)}\, dt,
\]
where we put $\tau^*_{0}=0, \tau^*_{r+1}=T$. Let $x_0$ be located on
some interval $(z_{k-1}, z_k)$, $k=2, \dots,m$, say, $x_0\in(z_1,
z_2)$. Then, on the interval $(0, \tau^*_{1})$, the trajectory $f$ is
contained in the segment
$[z_1, z_2]$. The functions $a_0, \sigma_0$ are constant and coincide
with $\bar a, \bar\sigma$ on $(z_1, z_2)$. In addition,
$a_0=a(z_1+)=a(z_2-)$, $\sigma_0=\sigma(z_1+)=\sigma(z_2-)$; hence, by
the choice of $\bar a, \bar\sigma$ we have
\[
{\bar a^2(z)\over\bar\sigma^{2}(z)}\leq{a^2_0(z)\over\sigma
^{2}_0(z)}, \quad z=z_1,
z_2.
\]
Then by \eqref{zero} with $A=\{z_1, z_2\}$ we have
\[
\begin{aligned}\int_0^{\tau^*_{1}}
{(\dot f_t-\bar a(f_t))^2\over\bar
\sigma
^{2}(f_t)}\, dt&=\int_0^{\tau^*_{1}} \biggl(
{(\dot f_t- a_0(f_t))^2\over
\sigma^{2}_0(f_t)}1_{f_t\not\in A}+ {\bar a^2(f_t)\over\bar\sigma
^{2}(f_t)}
1_{f_t \in A} \biggr)\, dt
\\
&\leq\int_0^{\tau^*_{1}} \biggl({(\dot f_t- a_0(f_t))^2\over\sigma
^{2}_0(f_t)}1_{f_t\not\in A}+
{a^2_0(f_t)\over\sigma^{2}_0(f_t)} 1_{f_t \in A} \biggr)\, dt
\\
&=\int_0^{\tau^*_{1}}{(\dot f_t- a_0(f_t))^2\over\sigma
^{2}_0(f_t)}\, dt=\int
_0^{\tau^*_{1}}{(\dot y_t^0- a_0(y_t^0))^2\over
\sigma^{2}_0(y_t^0)}. \end{aligned}
\]
Analogous inequalities hold on each of the time intervals $(\tau^*_{n},
\tau^*_{n+1}),$ $n=1, \dots, r$, with $a_0,\sigma_0$ changed to
$\bar a_{\kappa^*_{n}},\bar\sigma_{\kappa^*_{n}}$ (the proof is
similar and omitted). Thus,
\begin{equation}
\label{compare} %
\begin{aligned} I(f)&\leq{1\over2}\int
_0^{\tau^*_{1}}{(\dot y_t^0-
a_0(y_t^0))^2\over\sigma^{2}_0(y_t^0)}\, dt+
{1\over2}\sum_{n=1}^{r}\int
_{\tau^*_{n}}^{\tau^*_{n+1}}{(\dot y_t^{\kappa^*_{n},n}-
\bar a_{\kappa^*_{n}}(y_t^{\kappa^*_{n},n}))^2\over\bar\sigma
^{2}_{\kappa^*_{n}}(y_t^{\kappa^*_{n},n})}\, dt
\\
&\leq{1\over2} \int_0^T
{(\dot y^0_t- a_0(y_t^0))^2\over\sigma
^{2}_0(y_t^0)}\, dt+{1\over2}\sum
_{k=1}^m\sum_{n=1}^N
\int_0^T{(\dot y^{k,n}_t- \bar a_k(y^{k,n}_t))^2\over\bar\sigma
^{2}_k(y_t^{k,n})}\, dt=J(y).
\end{aligned} %
\end{equation}
This gives a contradiction with inequalities $J(y)\leq c$ and $I(f)>c$,
which completes the proof of the second inequality in \eqref{verify}.

The first inequality in \eqref{verify} holds immediately for $f$ such
that $I(f)=\infty$. We fix $f$ with $I(f)<\infty$ and $\gamma>0$ and
construct $y_\gamma$ such that $F(y_\gamma)=f$, the functions $\tau
_1(\cdot), \tau_2(\cdot), \dots$ are continuous at $y_\gamma$, and
$J(y_\gamma)\leq I(f)+\gamma$. This completes the proof of \eqref{verify}.

The construction explained gives a cue for the choice of $y=y_\gamma$
(we omit the index $\gamma$ to simplify the notation). We put $y^0$
equal to $f$ until its first time moment $\tau_1^*$ of hitting the set
$\{z_1, z_2\}$ (we still assume that $x_0\in(z_1, z_2)$). Then we
extend $y^0$ to the entire time interval $[0,T]$, and we aim to make
the integral
\begin{equation}
\label{defect} \int_{\tau^*_{1}}^T{(\dot y_t^0- a_0(y_t^0))^2\over\sigma
^{2}_0(y_t^0)}
\, dt
\end{equation}
small enough; that is, to make small the error in the second inequality
in \eqref{compare}, which arises because of the integral of $y^0$.
If we put $y^0_t=y_{\tau_1^*}+a_0(t-\tau_1^*)$, then we obtain the
trajectory at which the integral \eqref{defect} equals zero; we call
such a~trajectory a \emph{zero-energy} one. However, under such a
choice, we may fail with the other our requirement that $\tau_1(\cdot)$
should be continuous at the point $y$. It is easy to verify that for
such a continuity, it suffices that $y^0$, if hitting $\{z_1, z_2\}$ at
a point, say, $z_1$ at every interval $(\tau_1^*,\tau_1^*+\delta),\
\delta>0$, takes values both from $(-\infty, z_1)$ and $(z_1, \infty)$.
We can perturb the zero-energy trajectory introduced above on a small
time interval near $\tau_1^*$ in such a way that this new trajectory
possesses the continuity property explained before, and the
integral~\eqref{defect} is~$\leq\gamma/N$.

Then we iterate this procedure. Observe that, for any $k$, by the
construction of the function $\bar a_k$ there exists at least one
corresponding zero-energy trajectory with the initial value $z_k$,
which now is defined as a solution to the ODE
\[
\dot g_t=\bar a_k(g_t), \quad t>0,\quad g_0=z_k.
\]
We have $\kappa^*_1$ uniquely determined by the trajectory $f$ (in
fact, by the part of this trajectory up to time $\tau_1^*$). For
$k\not
=\kappa^*_1$, we define $y^{k,1}$ as the zero-energy trajectory on
$[0,T]$ that starts from $x_k$ and corresponds to the coefficient $\bar
a_k$. All these trajectories are ``phantom'' in the sense that they
neither are involved into the representation of $f$ through $y$ nor
give an impact into $J(y)$. For
$k=\kappa^*_1$, we define $y^{k,1}$ similarly as before: it equals
$f_{t+\tau_1^*}$ for $t\leq\tau_2^*-\tau_1^*$, and afterwards it is
defined as a perturbation of a zero-energy trajectory that makes $\tau
_2(\cdot)$ continuous in $y$ and
\[
\int_{\tau_2^*-\tau^*_{1}}^T{(\dot y_t^{k,1}- \bar
a_k(y_t^{k,1}))^2\over\bar\sigma^{2}_k(y_t^{k,1})}\, dt \leq
{\gamma
\over N}.
\]
Repeating this construction $\leq N$ times, we finally get the required
function \mbox{$y=y_\gamma$}. This completes the proof of \eqref{verify} and
thus of \eqref{wLDP}. Together with the exponential tightness proved in
Section \ref{s31}, this completes the proof of the LDP in this case.

\subsection{Case III. Piecewise constant $a/\sigma^2$, general
$\sigma$}\label{s34} In this section, we remove the assumption on $a,\sigma$
to be piecewise constant, still keeping this assumption for $a/\sigma
^2$; we also assume that $a, \sigma$ are constant on $(-\infty, R]$ and
$[R,\infty)$ for some $R$. Our basic idea is to represent $\{X^\eps\}$
as the image under a time changing transformation of a family $\{Y^\eps
\}$ and then to use the semicontraction principles. The same approach
was used in \cite{kulik-sobolieva}, where the LDP was established for a
solution of \eqref{sde} with $a\equiv0$; in this case, $Y^\eps$ was
taken in the form $Y^\eps_t=x_0+\eps W_t$. In our current setting, the
choice of the coefficients for the SDE that defines $Y^\eps$ should
take into account the common discontinuity points for $a/\sigma$ and
$\sigma$. This becomes visible both from an analysis of the proof of
Theorem 1 in \cite{kulik-sobolieva} and from the definition of the
functions $\bar a, \bar\sigma$, which combines the left- and
right-hand values of both $a$ and~$\sigma$ at the discontinuity points.
The proper choice of the family is explained below. Some parts of the
arguments are similar to those in \cite{kulik-sobolieva}. We omit
detailed proofs
whenever it is possible to give a reference to \cite{kulik-sobolieva}
and focus on the particularly new points.

We assume $a/\sigma^2$ to be piecewise constant with discontinuity
points $z_1<\dots<z_m$ and put (with the convention $\prod_\varnothing=1$)
\[
\tilde\sigma(x)=\prod_{k:z_k\leq x}{\sigma(z_k+)\over\sigma(z_k-)},\qquad
\upsilon(x)={\tilde\sigma(x)\over\sigma(x)}, \qquad
\tilde a(x)=a(x)\upsilon^{2}(x).
\]
Under such a choice, $\tilde\sigma=\sigma\upsilon$, and thus the
function $\tilde a/\tilde\sigma^2$ equals $a/\sigma^2$ and is constant
on each of the intervals $(-\infty, z_1), \dots, (z_m, \infty)$. By
construction, $\tilde\sigma$ is constant on these intervals as well;
hence, $\tilde a, \tilde\sigma$ fit the case studied in the previous
section, and the required LDP holds for the family $Y^\eps$ of the
solutions to \eqref{sde} with these coefficients and $Y^\eps_0=x_0$.
This construction yields also the following property, which will be
important below: the function $a=(a/\sigma^2)\sigma^2$ does not change
its sign on each of the intervals $(-\infty, z_1), \dots, (z_m,
\infty
)$. Hence, denoting
$B=a^2/\sigma^2$ and $\bar B=\bar a^2/\bar\sigma^2$, we get
\begin{equation}
\label{BBbar} \bar B(z)=\min\bigl(B(z-), B(z+)\bigr)=\min\bigl(\bar B(z-), \bar
B(z+)\bigr), \quad z\not \in\{z_k\}.
\end{equation}

Fix $\eps>0$ and define
\[
\eta_t=\int_0^t
\upsilon^2 \bigl(Y^\eps_s \bigr)\, ds, \quad t
\geq0,
\]
$\tau_t=[\eta]^{-1}_t$ (the inverse function w.r.t. $t$), and $X^\eps
_t=Y^\eps_{\tau_t}$. Then $X^\eps$ is a weak solution to \eqref{sde}
with $X_0^\eps=x_0$; see \cite{watanabe-ikeda}, IV \S7.

In the above construction, $\eta_t\geq c^2 t$ and thus $\tau_t\leq
c^{-2}t$; see \eqref{coef}. We put $\tilde T=c^{-2}T$, $\Y=C(0,
\tilde
T)$, and
define $Y^\eps$ as a family of solutions to \eqref{sde} with the
coefficients $\tilde a, \tilde\sigma$ and the time horizon $\tilde T$.
Then the family $X^\eps$ possesses a~representation $X^\eps=F(Y^\eps)$
with the mapping $F:\Y\to\X$ defined by
\[
\bigl[F(y) \bigr]_t=y_{\tau_t(y)}, \qquad\tau_t(y)=
\bigl[\eta(y) \bigr]^{-1}_t, \quad t\in[0, T],
\]
\[
\eta_t(y)=\int_0^t
\upsilon^2(y_s)\, ds,\quad t\in[0, \tilde T].
\]
Observe that for $F$ to be continuous at a point $y\in\Y$, it suffices
that $y$ spends zero time in the set $\varDelta_\upsilon$ of the
discontinuity points of the function $\upsilon$; see \cite
{kulik-sobolieva}, Lemma 1 and Corollary 1. Now $\varDelta_\upsilon
\subset
\varDelta_a\cup\varDelta_\sigma$ is at most countable, and it is easy to see
that the continuity set $\varLambda_F$ has probability 1 w.r.t.\ the
distribution of each $Y^\eps$, that is, we can apply Lemma \ref{l1}.

Our further aim is to prove \eqref{verify} in the above setting, which
then would imply \eqref{wLDP} and thus prove the LDP. The general idea
of the proof is similar to that of Theorem 1 in \cite{kulik-sobolieva},
though particular technicalities differ substantially.

First, for a given $f\in\X$, we describe explicitly the set $F^{-1}(\{
f\})$. We put
\[
\zeta_t(f)=\int_0^t
\upsilon^{2}(f_s)\, ds,\quad t\in[0, T].
\]
If $f=F(y)$, then
\[
\zeta_t(f)=\int_0^t
\upsilon^{2}(y_{\tau_{s}(y)})\, ds=\int_0^{\tau
_t(y)}
\upsilon^{2}(y_r) \upsilon^{-2}(y_r)
\, dr=\tau_t(y), \quad t\in[0, T];
\]
here we changed the variables $r=\tau_s(y)$ and used that
\[
dr=\tau_s'(y)\,ds={1\over\upsilon^2(y_{\tau_s(y)})}\,ds=
{1\over\upsilon
^2(y_{r})}\, ds.
\]
Therefore,
\[
f_t=y_{\tau_t(y)}=y_{\zeta_t(f)}, \quad t\in[0, T].
\]
Observe that $\zeta_T(f)\leq c^{-2}T=\tilde T$ and define
\[
\pi_t(f)= \bigl[\zeta(f) \bigr]^{-1}_t=\inf
\bigl\{r: \zeta_r(f)\geq t \bigr\}, \quad t\in \bigl[0,
\zeta_T(f) \bigr].
\]
Then we conclude that
\begin{equation}
\label{resolve} y_t=f_{\pi_t(f)}, \quad t\in\bigl[0,
\zeta_T(f)\bigr];
\end{equation}
that is, for any $y\in F^{-1}(\{f\})$, the part of its trajectory with
$t\leq\zeta_T(f)$ is uniquely defined. On the other hand, it is easy
to show that \emph{any} $y\in\Y$ satisfying \eqref{resolve} belongs to
$F^{-1}(\{f\})$.

Next, we denote by $\hat a, \hat\sigma$ the modified coefficients,
which correspond to the coefficients $\tilde a, \tilde\sigma$ in the
sense explained in Section \ref{s1}. Since $\tilde a=a \upsilon^2$ and
$ \tilde\sigma=\sigma\upsilon$, we easily see that
\begin{equation}
\label{eq} \hat a(x)=\bar a(x)\upsilon^2(x),\qquad\hat\sigma(x)=
\bar\sigma (x)\upsilon(x)
\end{equation}
at every continuity point $x$ for $\upsilon$. Then, for any $y\in \mathit{AC}(0,
\tilde T)$ with $y_0=x_0$ that spends zero time in the set $\varDelta
_\upsilon$, we have
\[
J(y)={1\over2}\int_0^{\tilde T}
{(\dot y_t-\hat a(y_t))^2\over\hat
\sigma^{2}(y_t)}\, dt={1\over2}\int_0^{\tilde T}
{(\dot y_t-\bar
a(y_t)\upsilon^{2}(y_t) )^2\over\bar\sigma^{2}(y_t)\upsilon
^{2}(y_t)}\, dt.
\]
On the other hand, using \eqref{resolve} and making the time change
$s=\pi_t(f)$ with $f=F(y)$, we get
\[
\begin{aligned}
&
{1\over2}\int_0^{\zeta_T(F(y))}
{(\dot y_t-\bar a(y_t)\upsilon
^{2}(y_t) )^2\over\bar\sigma^{2}(y_t)\upsilon^{2}(y_t)}\, dt
\\
&\quad{}= {1\over2}\int_0^{\zeta_T(F(y))}
{(\dot y_t\upsilon^{-2}(y_t)-\bar
a(y_t))^2\over\bar\sigma^{2}(y_t)}\upsilon^{2}(y_t)\, dt
\\
&\quad{}={1\over2}\int_0^{T}
{(\dot f_s-\bar a(f_s))^2\over\bar\sigma
^{2}(f_s)}\, ds=I \bigl(F(y) \bigr)
\end{aligned} %
\]
because now $t=\zeta_s(f)=\tau_s(y)$ and thus
\[
dt={ds\over\upsilon^{2}(y_t)}, \qquad\dot y_t\upsilon
^{-2}(y_t)=(y_{\tau_s(y)})'_s=
\dot f_s.
\]
Thus,
\begin{equation}
\label{Jdec} J(y)=I\bigl(F(y)\bigr)+J^{\mathit{tail}}(y)
\end{equation}
with
\[
J^{\mathit{tail}}(y)={1\over2}\int_{\zeta_T(F(y))}^{\tilde T}
{(\dot y_t-\hat
a(y_t))^2\over\hat\sigma^{2}(y_t)}\, dt.
\]

Now we are ready to proceed with the proof of the first inequality in
\eqref{verify}.
\begin{lem} For any $f\in C(0,T)$,
\begin{equation}
\label{low} I_{\mathit{lower}}(f)\leq I(f).
\end{equation}
\end{lem}
\begin{proof} We consider only $f$ such that $I(f)<\infty$; otherwise,
the required inequality is trivial. Let us fix a function $y$
corresponding to $f$ by the following convention: it is given by
identity \eqref{resolve} up to the time moment $t=\zeta_T(f)$ and
follows a zero-energy trajectory afterward, that is, satisfies
\[
\dot y_t=\hat a(y_t)
\]
a.e.\ w.r.t.\ to the Lebesgue measure. We note that at least one such
zero-energy trajectory exists (it may be nonunique, and in this case,
we just fix one of such trajectories). Indeed, by construction, $\tilde
a$ is piecewise constant, so that the corresponding $\hat a$ is
piecewise constant as well. The proper choice of $\hat a(z_k)$ at those
points $z_k$ where $\tilde a(z_k-)>0, \tilde a(z_k+)<0$ yields that the
above ODE, which determines a zero-energy trajectory, admits at least
one solution.

If $f$ spends zero time in the set $\varDelta_\upsilon$ of discontinuity
points for $\upsilon$, then the same property holds for the
corresponding $y$ constructed above. Indeed, the first part of the
trajectory $y$ is just the time-changed trajectory $f$, and the second
part is a zero-energy trajectory. The latter trajectory is piecewise
linear, and we can separate a finite set of time intervals where it
either (a) moves with a constant speed $\not=0$ (and thus spends a zero
time in the set $\varDelta_\upsilon$, which has zero Lebesgue measure) or
(b) stays constant (in this case, it equals $z_k$ for some $k$, and, by
construction, $\upsilon$ is continuous at $\{z_k\}$). Hence, we
conclude that \eqref{Jdec} holds and, moreover, $J^{\mathit{tail}}(y)=0$, that
is, $J(y)=I(f)$. In addition, $y\in\varLambda_F$, which gives for this
$f$ the required inequality
\[
I_{\mathit{lower}}(f)\leq I(f).
\]

For a general $f$, we will show that, for each $\delta>0$, there exists
$f^\delta$ such that $f^\delta\in B_\delta(f)$, $I(f^\delta)\leq
I(f)+\delta$, and $f^\delta$ spends zero time in $\varDelta_\upsilon$;
since $I_{\mathit{lower}}$ is known to be lower semicontinuous, this will
complete the proof of the first inequality in \eqref{verify}. Recall
the decomposition $I=I_1+I_2+I_3$ from Section~\ref{s22} and note that
$I_2(f^n)\to I_2(f)$ if $f^n\to f$ in the uniform distance and
$I_1(f^n)\to I_1(f)$ if $f^n\to f$ in the distance
\[
d_\varSigma \bigl(f^1,f^2 \bigr)= \Biggl(\int_0^T
\bigl| \bigl({\varSigma \bigl(f_t^1
\bigr)} \bigr)'- \bigl({\varSigma \bigl(f_t^2
\bigr)} \bigr)'\bigr|^2\, dt \Biggr)^{1/2}.
\]
Hence, our aim is to construct a function $f^\delta$ that is close to
$f$ both in the uniform distance and in $d_\varSigma$, spends zero time in
the set $\varDelta_\upsilon$, and
\[
I_3 \bigl(f^\delta \bigr)\leq I_3(f)+\delta/3.
\]

We decompose the time set
\[
Q= \bigl\{t\in(0, T): f_t\not\in\{z_k\} \bigr\}
\]
into a disjoint union of open intervals and modify the function $x$ on
each of these intervals. On the complement to this union, the function
$f^\delta$ will remain the same; note that $\upsilon$ is continuous at
every point $z_k$, and hence in order to get a function that spends
zero time in $\varDelta_\upsilon$, it suffices to modify $f$ on $Q$ only.
In what follows, we fix an interval $(\alpha, \beta)$ from the
decomposition of the set $Q$ and describe the way to modify $f$ on
$(\alpha, \beta)$. The construction below is mostly motivated by
\eqref{BBbar}.
We fix some $\gamma>0$ and choose a finite partition $\{u_j\}$ of the
set $\{f_t, t\in[\alpha, \beta]\}$ such that the oscillation of the
function $\bar\sigma^2$ on each interval $(u_{j-1}, u_j)$ does not
exceed $\gamma$. Then there exists a finite partition $\alpha
=t_0<\dots
<t_m=\beta$ such that, on each time segment
$[t_{i-1}, t_i]$, the function~$x$ visits at most one point from the
set $\{u_j\}$. Then, on each interval $[t_{i-1}, t_i]$, we consider the
family
\[
f^{i,\kappa}_t=f_t+s_i\kappa
\phi_t^i,\quad\kappa>0,
\]
where $\phi^i$ is a function such that
\[
\phi_{t_{i-1}}^{i}=\phi_{t_{i}}^{i}=0, \qquad
\int_{t_{i-1}}^{t_i} \bigl(\dot\phi^i_t
\bigr)^2\, dt<\infty,
\]
and $s_i$ is defined by the following convention: $s_i=+1$ if $\bar B$
is right-continuous at the (unique) point from the set $\{u_j\}$ that
is visited by $f$ on $[t_{i-1}, t_i]$ or if $f$ does not visit this
set; otherwise, $s_i=-1$. If, in addition, $\dot\phi^i_t\not=0$ a.e.,
then for all $\kappa>0$ except at most countable set of points, we have
that $f^{i,\kappa}$ spends zero time in $\varDelta_\upsilon$ on the time
interval; see \cite{kulik-sobolieva}, Lemma 2. The choice of the sign
$s_i$ yields that, for $\kappa>0$ small enough,
\[
B \bigl(f^{i,\kappa}_t \bigr)\leq\bar B(x_t)+\gamma,
\quad t\in[t_{i-1}, t_i]\cap\varDelta_{B} .
\]
Then $\kappa>0$ can be chosen small enough and the same for all
intervals $[t_{i-1}, t_i]$, so that the corresponding function $\tilde
f^{\kappa}$, which coincides with $f^{i,\kappa}$ on $[t_{i-1}, t_i]$,
satisfies
\[
\int_\alpha^\beta\bar B \bigl(\tilde
x^{\kappa}_t \bigr)\, dt\leq\int_\alpha
^\beta\bar B(x_t)\, dt+2\gamma(\beta-\alpha).
\]
It is also easy to see that, in addition, the following inequalities
can be guaranteed by the choice of (small) $\kappa$:
\[
\sup_{t\in(\alpha, \beta)}
\bigl|\tilde f^{\kappa}_t-x_t\bigr|
\leq\gamma(\beta-\alpha), \qquad\int_\alpha^\beta\bigl|
\bigl({\varSigma \bigl(f^{\kappa}_t \bigr)} \bigr)'-
\bigl({\varSigma(f_t)} \bigr)'\bigr|^2\, dt\leq
\gamma(\beta-\alpha).
\]
Repeating the same construction on each interval from the partition for
$Q$, we get a function $\tilde f$ such that
\[
\|\tilde f-f\|\leq\gamma T, \qquad
d_\varSigma^2(\tilde f, f)\leq\gamma T, \qquad
I_3(\tilde f)\leq I_3(f)+\gamma T,
\]
and $\tilde f$ spends zero time in $\varDelta_\upsilon$. Taking in this
construction $\gamma>0$ small enough, we obtain the required function
$f^\delta=\tilde f$, which completes the proof of \eqref{low}.
\end{proof}
Recall that $B$ and $\bar B$ satisfy \eqref{BBbar}.
For the similar pair of functions $\tilde B=\tilde a^2/\tilde\sigma^2$
and $\hat B=\hat a^2/\hat\sigma^2$, we have even more: the functions
$\hat a, \hat\sigma$ are constant on each of the intervals $(-\infty,
z_1), \dots, (z_m, \infty)$; hence,
\[
\hat B(z)=\tilde B(z), \quad z\not\in\{z_k\}.
\]
On the other hand, since $\tilde a=a \upsilon^2$ and $\tilde\sigma
=\sigma\upsilon$, we have $B=\tilde B\upsilon^{-2}$, and thus
\[
B(z)=\hat B(z)\upsilon^{-2}(z), \quad z\not\in\{z_k\}.
\]
This yields, for $z\not\in\{z_k\}$,
\begin{equation}
\label{wow} \bar B(z)=\hat B(z)\min\bigl(\upsilon^{-2}(z-),
\upsilon^{-2}(z+)\bigr).
\end{equation}
Recall that $\upsilon$ is continuous at each point $z_k$; hence, by
\eqref{eq} identity \eqref{wow} holds for all $z\in\Re$.

Now we are ready to proceed with the proof of the second inequality in
\eqref{verify}.
\begin{lem} For any $f\in C(0,T)$,
\begin{equation}
\label{up} I(f)\leq I_{\mathit{upper}}(f).
\end{equation}
\end{lem}

\begin{proof} Assuming \eqref{up} to fail for some $f$, we will have
sequences $\{y^n\}, \{\tilde y^n\}$ such that $\{\tilde y_n\}\subset
\varLambda_F$,
\[
F \bigl(\tilde y^n \bigr)\to f, \qquad
\bigl\|y^n-\tilde y^n\bigr\|\to0, \quad\mbox{and}\quad\limsup_n J
\bigl(y^n \bigr)<I(f).
\]
Then $\{y^n\}$ belongs to some level set $\{J(y)\leq c\}$ of a good
rate function $J$. Hence, passing to a subsequence, we can assume that
both $\{y^n\}$ and $\{\tilde y^n\}$ converge to some $y\in\Y$. In
addition, $J(y)\leq\liminf_n J(y^n)<I(f)$.

Next, denote
\[
\tau_t^n=\tau_t \bigl(\tilde y^n
\bigr),
\]
where $\tau(\cdot)$ is the function introduced in the definition of
$F$. Then each $\tau^n\in \mathit{AC}(0, T)$ with its derivative taking values
from $[C^{-2}, c^{-2}]$; see \eqref{coef}. This allows us, passing to a
subsequence, assume that there exists a uniform limit $\tau=\lim_{n}\tau
^n$ and that $\dot\tau^n\to\dot\tau$ weakly in $L_2(0,T)$.

Observe that
\[
\sup_{t\in[0, T]}
\bigl|F \bigl(\tilde y^n\bigr)_t-y_{\tau^n_t}\bigr|
=
\sup_{t\in[0,T]}
\bigl|\tilde y^n_{\tau^n_t}-y_{\tau^n_t}\bigr|
\to0.
\]
Thus,
\[
f_t=y_{\tau_t}, \quad t\in[0,T].
\]
Then we have a representation for the part of the trajectory $y$
similar to \eqref{resolve}:
\begin{equation}
\label{resolveY} y_s=f_{\eta_s}, \qquad
s\in[0,\tau_T],\qquad
\eta:=\tau^{-1}.
\end{equation}
We observe that
\[
J(y)\geq{1\over2}\int_0^{\tau_T}
{(\dot y_s-\hat a(y_s))^2\over\hat
\sigma^{2}(y_s)}\, ds
\]
and give a decomposition for the latter integral, similar to \eqref
{repI}. Recall that the function $\hat a/\hat\sigma^2$ coincides with
$\bar a/\bar\sigma^2$ at each point except the finite set $\{z_k\}$. Then
\[
S(x)=\int_0^x {\bar a(z)\over\bar\sigma^{2}(z)}\, dz=\int
_0^x {\hat
a(z)\over\hat\sigma^{2}(z)}\, dz, \quad x\in
\Re,
\]
and thus
\begin{equation}
\label{repJ} J(y)\geq {1\over2}\int_0^{\tau_T}
{(\dot y_s)^2\over\hat\sigma^{2}(y_s)}\, ds+ \bigl[S(y_{\tau_T})-S(y_0)
\bigr]+{1\over2}\int_0^{\tau_T}
{\hat
a^2(y_s)\over\hat\sigma^{2}(y_s)}\, ds=:J_1+J_2+J_3.
\end{equation}
Now we will use \eqref{resolveY} in order to compare $J_i, i=1,2,3$,
with $I_i(f), i=1,2,3$. We have directly that $J_2=I_2(f)$. Next, we
change the variables $s=\tau_t$,\ and get
\[
J_3={1\over2}\int_0^{T}
{\hat a^2(f_t)\over\hat\sigma^{2}(f_t)}\dot\tau_t\, dt.
\]
Recall that we assumed $\dot\tau$ to be the $L_2$-weak limit of
\[
\dot\tau_t^n={1\over\upsilon^2(\tilde y^n_{\tau_t^n})}.
\]
On the other hand, $\tilde y^n_{\tau_t^n}\to f_t$, and then it is easy
to show that, for a.a. $t\in[0,T]$,
\begin{equation}
\label{dot} {1\over\max(\upsilon^2(f_t-), \upsilon^2(f_t+))}\leq\dot\tau _t\leq
{1\over\min(\upsilon^2(f_t-), \upsilon^2(f_t+))}.
\end{equation}
Then by \eqref{wow} we get
\[
J_3\geq{1\over2}\int_0^{T}
{\hat a^2(f_t)\over\hat\sigma^{2}(f_t)} {1\over\max(\upsilon
^2(f_t-), \upsilon^2(f_t+))}\, dt={1\over2}
\int_0^{T} \bar B(f_t)\,
dt=I_3(f).
\]
Finally, changing the variables $s=\tau_t$, we get
\[
J_1={1\over2}\int_0^{T}
{(\dot f_t)^2\over\hat\sigma^{2}(f_t)\dot\tau
_t}\, dt.
\]
Denote $Q=\{t\in[0,T]: f_t\in\varDelta_\upsilon\}$\ and recall that
because $\varDelta_\upsilon$ has zero Lebesgue measure, $\dot f_t=0$ for
a.a. $t\in Q$. On the other hand, if $f_t\not\in\varDelta_\upsilon,$
then by \eqref{eq} and \eqref{dot} we have
\[
{1\over\hat\sigma^{2}(f_t)\dot\tau_t}={1\over\bar\sigma^{2}(f_t)};
\]
thus,
\[
J_1={1\over2}\int_{[0,T]\setminus Q}
{(\dot f_t)^2\over\hat\sigma
^{2}(f_t)\dot\tau_t}\, dt={1\over2}\int_{[0,T]\setminus Q}
{(\dot
f_t)^2\over\bar\sigma^{2}(f_t)}\, dt=I_1(f).
\]
Summarizing the above, we get $J(y)\geq I(f)$, which contradicts to the
assumption made at the beginning of the proof.
\end{proof}

\subsection{Completion of the proof: general $a, \sigma$}\label{s35}
In this last part of the proof, we remove the assumption $a/\sigma^2$
to be piecewise constant and prove the required statement in the full
generality. According to Section \ref{s21}, it suffices to prove that, for
fixed $f\in C(0, T)$ and $\varkappa>0$,
\begin{equation}
\label{wLDPup} \limsup_{\eps\to0}\eps^{2}\log\P
\bigl(X^\eps\in B_\delta(f)\bigr)\leq -I(f)+\varkappa
\end{equation}
for {some} $\delta>0$ and
\begin{equation}
\label{wLDPlow} \liminf_{\eps\to0}\eps^{2}\log\P
\bigl(X^\eps\in B_\delta(f)\bigr)\geq -I(f)-\varkappa
\end{equation}
for {each} $\delta>0$. While doing that, we can and will assume that
$a, \sigma$ are constant on $(-\infty, R]$ and $[R,\infty)$ for some $R$.

Consider, together with the original SDE \eqref{sde}, the SDE
\begin{equation}
\label{sde2} dY_t^\eps=\bigl[a\bigl(Y_t^\eps
\bigr)+\alpha\bigl(Y_t^\eps\bigr)\sigma\bigl(Y^\eps_t
\bigr)\bigr]\, dt+\eps \sigma\bigl(Y^\eps_t\bigr)
dW_t, \quad Y^\eps_0=x_0\in\Re,
\end{equation}
where $\alpha$ is a bounded function to be specified later. Then by the
Girsanov theorem
\[
\P \bigl(X^\eps_\cdot\in A \bigr)=\E1_{Y^\eps\in A}
\mathcal{E}_T^\eps,
\]
\[
\mathcal{E}_T^\eps:=\exp \Biggl[-\eps^{-1}\int
_0^T\alpha \bigl(Y_t^\eps
\bigr)\, dW_s-{\eps^{-2}\over2}\int_0^T
\alpha^2 \bigl(Y_t^\eps \bigr)\, ds \Biggr];
\]
see \cite{watanabe-ikeda}, Chapter IV, Theorem 4.2. Then, for arbitrary
$p,q>1: 1/p+1/q=1$, we have
\[
\P \bigl(X^\eps_\cdot\in A \bigr)\leq\P
\bigl(Y^\eps_\cdot \in A \bigr)^{1/p} \bigl( \E
\bigl(\mathcal{E}_T^\eps \bigr)^q
\bigr)^{1/q}.
\]
Let $|\alpha(x)|\leq\gamma$. Then we have
\[
\begin{aligned}\E \bigl(\mathcal{E}_T^\eps
\bigr)^q&=\E \exp \Biggl[-q\eps^{-1}\int
_0^T\alpha \bigl(Y_t^\eps
\bigr)\, dW_s-q {\eps^{-2}\over2}\int_0^T
\alpha^2 \bigl(Y_t^\eps \bigr)\, ds \Biggr]
\\
&=\E\exp \Biggl[ \bigl(q^2-q \bigr){\eps^{-2}\over2}\int
_0^T\alpha^2 \bigl(Y_t^\eps
\bigr)\, ds \Biggr] \mathcal{E}_T^{q\eps}\leq e^{(q^2-q)\gamma
^2\eps^{-2}/2}
\E\, \mathcal{E}_T^{q\eps}. \end{aligned} %
\]
Since $\alpha$ is bounded, $\mathcal{E}^{q\eps}$ is a martingale. Thus,
we can summarize the above calculation as follows:
\[
\log\P \bigl(X^\eps_\cdot\in A \bigr)\leq
{1\over p} \log\P \bigl(Y^\eps_\cdot\in A
\bigr)+ {(q-1)\gamma^2\eps
^{-2}\over2}.
\]

In what follows, we will choose $\alpha$ such that the function
$(a+\alpha\sigma)/\sigma^2$ is piecewise constant. Then the result
proved in the previous section will provide that, for given $f\in
C(0,T)$ and $\kappa>0$, there exists $\delta>0$ such that
\begin{equation}
\label{12} \limsup_{\eps\to0}\eps^{2}\log\P
\bigl(Y^\eps\in B_\delta(f)\bigr)\leq -\tilde I(f)+
{\kappa\over4},
\end{equation}
where $\tilde I$ is the rate functional that corresponds to the new
drift coefficient $a+\alpha\sigma$ and the same diffusion coefficient.
It is easy to verify using the representation \eqref{repI} and its
analogue for $\tilde I$ that we can choose $\gamma$ small enough so
that the above construction with arbitrary $\alpha$ such that $\|
\alpha
\|\leq\gamma$ yields
\begin{equation}
\label{13} \tilde I(f)\geq I(f)-{\kappa\over4}.
\end{equation}
In that case, we get
\[
\limsup_{\eps\to0}\eps^{2}\log\P \bigl(X^\eps
\in B_\delta(f) \bigr)\leq-{1\over
p} I(f)+
{\kappa\over4p}+{(q-1)\gamma^2\over2}.
\]
Now we are ready to summarize the entire argument. For given $f\in
C(0,T)$ and $\kappa>0$, we take $p>1$ close enough to 1 such that
\[
{1\over p} I(f)\geq I(f)-{\kappa\over4}.
\]
Then we take $\gamma>0$ small enough such that \eqref{13} holds and
\[
{(q-1)\gamma^2\over2}={\gamma^2\over2(p-1)}\leq{\kappa\over4}.
\]

Observe that the function $a/\sigma^2$ has left and right limits at
every point and is constant on
$(-\infty, R]$ and $[R,\infty)$. Then it can be approximated by
piecewise constant functions in the uniform norm. This means that we
can find $\alpha$ with $\|\alpha\|\leq\gamma$ such that the function
$(a+\alpha\sigma)/\sigma^2$ is piecewise constant. Then there exists
$\delta>0$ such that \eqref{12} holds, and we finally deduce
\[
\limsup_{\eps\to0}\eps^{2}\log\P \bigl(X^\eps
\in B_\delta(f) \bigr)\leq-I(f)+{\kappa\over4}+
{\kappa\over4p}+{\kappa\over4}+{\kappa\over
4}<-I(f)+
\kappa,
\]
which completes the proof of \eqref{wLDPup}.

Exactly the same argument provides the proof of \eqref{wLDPlow} as
well, with the minor change that now the law of $Y^\eps$ should be
expressed in the terms of $X^\eps$; that is, we should use the Girsanov
theorem in the following form:
\[
\P \bigl(Y^\eps_\cdot\in A \bigr)=\E1_{X^\eps\in A}
\mathcal{H}_T^\eps,
\]
\[
\mathcal{H}_T^\eps:=\exp \Biggl[\eps^{-1}\int
_0^T\alpha \bigl(X_t^\eps
\bigr)\, dW_s-{\eps^{-2}\over2}\int_0^T
\alpha^2 \bigl(X_t^\eps \bigr)\, ds \Biggr].
\]
The rest of the proof remains literally the same; we omit the detailed
exposition.

\section*{Acknowledgements} The authors express their gratitude to the
referee for useful remarks and to A.~Pilipen\-ko for highly instructive
discussion. which led us to a significant improvement of the presentation.


%

%
\end{document}